\documentclass[11pt]{amsart}
\usepackage{amsfonts,amscd}
\usepackage{amssymb}
\usepackage{url}
\usepackage{graphicx}
\usepackage[english]{babel}

\usepackage{color} 

\theoremstyle{plain}
\newtheorem{theorem}                {Theorem}      [section]
\newtheorem{proposition}  [theorem]  {Proposition}
\newtheorem{corollary}    [theorem]  {Corollary}
\newtheorem{lemma}        [theorem]  {Lemma}

\theoremstyle{definition}
\newtheorem{example}      [theorem]  {Example}
\newtheorem{remark}       [theorem]  {Remark}

\def \R{{\mathbb R}}
\def \s{{\mathbb S}}

\def \n{{\mathbb N}}
\def \H{{\mathbb H}}
\def \C{{\mathbb C}}

\usepackage{color}

\def \link {~}
\def \1 {\`}

\DeclareMathOperator{\grad}{grad}
\DeclareMathOperator{\trace}{Tr}
\numberwithin{equation}{section}

\title[Polyharmonic  pseudo-Riemannian hypersurfaces]{Polyharmonic hypersurfaces into pseudo-Riemannian space forms}

\author{V. Branding}
\address{University of Vienna, Faculty of Mathematics\\
Oskar-Morgenstern-Platz 1, 1090 Vienna, Austria}
\email{volker.branding@univie.ac.at}

\author{S.~Montaldo}
\address{Universit\`a degli Studi di Cagliari\\
Dipartimento di Matematica e Informatica\\
Via Ospedale 72\\
09124 Cagliari, Italia}
\email{montaldo@unica.it}

\author{C.~Oniciuc}
\address{Faculty of Mathematics\\ ``Al.I. Cuza'' University of Iasi\\
Bd. Carol I no. 11 \\
700506 Iasi, Romania}
\email{oniciucc@uaic.ro}

\author{A.~Ratto}
\address{Universit\`a degli Studi di Cagliari\\
Dipartimento di Matematica e Informatica\\
Via Ospedale 72\\
09124 Cagliari, Italia}
\email{rattoa@unica.it}

\usepackage{hyperref}
\begin{document}
\begin{abstract}
In this paper we shall assume that the ambient manifold is a pseudo-Riemannian space form $N^{m+1}_t(c)$ of dimension $m+1$ and index $t$ ($m\geq2$ and $1 \leq t\leq m$). We shall study hypersurfaces $M^{m}_{t'}$ which are polyharmonic of order $r$ (briefly, $r$-harmonic), where $r\geq 3$ and either $t'=t$ or $t'=t-1$. Let $A$ denote the shape operator of $M^{m}_{t'}$. Under the assumptions that $M^{m}_{t'}$ is CMC and $\trace A^2$ is a constant, we shall obtain the general condition which determines that $M^{m}_{t'}$ is $r$-harmonic. As a first application, we shall deduce the existence of several new families of proper $r$-harmonic hypersurfaces with diagonalizable shape operator, and we shall also obtain some results in the direction that our examples are the only possible ones provided that certain assumptions on the principal curvatures hold. Next, we focus on the study of isoparametric hypersurfaces whose shape operator is non-diagonalizable and also in this context we shall prove the existence of some new examples of proper $r$-harmonic hypersurfaces ($r \geq 3$). Finally, we shall obtain the complete classification of proper $r$-harmonic isoparametric pseudo-Riemannian surfaces into a $3$-dimensional Lorentz space form.
\end{abstract}



\thanks{
The author V.B. acknowledges the support of the Austrian Science Fund (FWF) through the START-Project Y963-N35 of Michael Eichmair. The authors S.M. and A.R. were supported by Fondazione di Sardegna (project STAGE) and Regione Autonoma della Sardegna (Project KASBA)}

\maketitle

\section{Introduction}\label{Intro}
We recall that, in the \textit{Riemannian case}, \textit{harmonic maps} are the critical points of the {\em energy functional}
\begin{equation}\label{energia}
E(\varphi)=\frac{1}{2}\int_{M}\,\|d\varphi\|^2\,dV \, ,
\end{equation}
where $\varphi:M\to N$ is a smooth map between two Riemannian
manifolds $(M^m,g)$ and $(N^n,h)$. In particular, $\varphi$ is harmonic if and only if it is a solution of the Euler-Lagrange system of equations associated to \eqref{energia}, i.e.,
\begin{equation}\label{harmonicityequation}
  - d^* d \varphi =   {\trace} \, \nabla d \varphi =0 \, .
\end{equation}
The left member of \eqref{harmonicityequation} is a vector field along the map $\varphi$ or, equivalently, a section of the pull-back bundle $\varphi^{-1} TN$: it is called {\em tension field} and denoted $\tau (\varphi)$. In addition, we recall that, if $\varphi$ is an \textit{isometric immersion}, then $\varphi$ is a harmonic map if and only if the immersion $\varphi$ defines a minimal submanifold of $N$ (see \cite{EL83, EL1} for background).

Next, in order to define the notion of an $r$-harmonic map, we consider the following family of functionals which represent a version of order $r$ of the classical energy \eqref{energia}. If $r=2s$, $s \geq 1$:
\begin{eqnarray}\label{2s-energia}
E_{2s}(\varphi)&=& \frac{1}{2} \int_M \, \langle \, \underbrace{(d^* d) \ldots (d^* d)}_{s\, {\rm times}}\varphi, \,\underbrace{(d^* d) \ldots (d^* d)}_{s\, {\rm times}}\varphi \, \rangle_{_N}\, \,dV \nonumber\\ 
&=& \frac{1}{2} \int_M \, \langle \,\overline{\Delta}^{s-1}\tau(\varphi), \,\overline{\Delta}^{s-1}\tau(\varphi)\,\rangle_{_N} \, \,dV\,.
\end{eqnarray}
In the case that $r=2s+1$:
\begin{eqnarray}\label{2s+1-energia}
E_{2s+1}(\varphi)&=& \frac{1}{2} \int_M \, \langle\,d\underbrace{(d^* d) \ldots (d^* d)}_{s\, {\rm times}}\varphi, \,d\underbrace{(d^* d) \ldots (d^* d)}_{s\, {\rm times}}\varphi\,\rangle_{_N}\, \,dV\nonumber \\
&=& \frac{1}{2} \int_M \,\sum_{j=1}^m \langle\,\nabla^\varphi_{e_j}\, \overline{\Delta}^{s-1}\tau(\varphi), \,\nabla^\varphi_{e_j}\,\overline{\Delta}^{s-1}\tau(\varphi)\, \rangle_{_N} \, \,dV \,.
\end{eqnarray}

Here, \(\overline{\Delta}=d^* d\) represents the Laplacian on the pull-back bundle $\varphi^{-1} TN$.
Then a map $\varphi:(M^m,g)\to(N^n,h)$ is \textit{$r$-harmonic} if, for all variations $\varphi_t$,
$$
\left .\frac{d}{dt} \, E_{r}(\varphi_t) \, \right |_{t=0}\,=\,0 \,\,.
$$
In the case that $r=2$, the functional \eqref{2s-energia} is called \textit{bienergy} and its critical points are the so-called \textit{biharmonic maps}. A very ample literature on biharmonic maps is available and we refer to \cite{Chen, Jiang, SMCO, Ou} for an introduction to this topic. More generally, the \textit{$r$-energy functionals} $E_r(\varphi)$ defined in \eqref{2s-energia}, \eqref{2s+1-energia} have been intensively studied (see \cite{Volker, Volker2, BMOR1, Maeta1, Maeta3, Maeta2, MOR-space forms, Mont-Ratto4, Na-Ura, Wang, Wang2}, for instance).

We say that an $r$-harmonic map is {\it proper} if it is not harmonic (similarly, an $r$-harmonic submanifold, i.e., an $r$-harmonic isometric immersion, is {\it proper} if it is not minimal). We point out that, as observed by Maeta in his series of papers {\cite{Maeta1, Maeta3, Maeta2}}, in general $r$-harmonic does not imply $r'$-harmonic for $r'>r$ unless the target manifold is flat.

In our recent work \cite{MOR-space forms} we proved some general results for $r$-harmonic hypersurfaces into space forms and deduced that the value of the integer $r$ plays a crucial role to generate geometric phenomena which differ substantially from the classical situation corresponding to the biharmonic and triharmonic cases. For instance, if $\ell \geq 3$, there exists no isoparametric hypersurface of $\s^{m+1}$ of degree $\ell$ which is proper biharmonic or triharmonic. By contrast, when $r \geq 5$, there are several examples of such hypersurfaces which are proper $r$-harmonic (see \cite{MOR-space forms}). From the point of view of the differential geometry of submanifolds, the difficulties which one encounters in studying the equations which define a general $r$-harmonic submanifold are huge. Therefore, a reasonable starting point is to focus on the case that the ambient is a space form $N^{m+1}(c)$  (here and below, $c$ denotes the sectional curvature) and study CMC hypersurfaces with constant squared norm $\|A\|^2$ of the shape operator. In this order of ideas, in the Riemannian case we obtained the following general result:  
\begin{theorem}\label{Th-existence-hypersurfaces-c>0-e-c<0}\cite{MOR-space forms} Let $M^m$ be a non-minimal CMC hypersurface in a Riemannian space form $N^{m+1}(c)$ and assume that $\|A\|^2$ is constant. Then
$M^m$ is proper $r$-harmonic ($r \geq 3$) if and only if
\begin{equation*}\label{r-harmonicity-condition-in-general}
\|A\|^4-m\,c\,\|A\|^2-(r-2)m^2 \,c\, \alpha^2=0 \,,
\end{equation*}
where the constant $\alpha$ denotes the mean curvature of $M^m$.
\end{theorem} 
\begin{remark} In the biharmonic case, a similar result is available under a less restrictive hypothesis. Indeed,
\begin{theorem}\label{Th-bihar-hypersurfaces-spheres}
(See \cite{CMO02, CMO03, Jiang}) Let $M^m$ be a non-minimal CMC hypersurface in $N^{m+1}(c)$. Then $M^m$ is proper biharmonic if and only if $\|A\|^2=cm$. In particular, if $c \leq 0$, then no such $M^m$ can exist. 
\end{theorem}
\end{remark}

The first goal of this paper is to establish a version of Theorem\link\ref{Th-existence-hypersurfaces-c>0-e-c<0} when the ambient is a pseudo-Riemannian space form. Next, we shall describe several new examples and some geometric applications. In order to state our results, it is convenient to introduce first some basic notions concerning pseudo-Riemannian manifolds and geometry. Therefore, our paper is organised as follows. 

In Section\link\ref{Sec-preliminaries} we shall review some basic aspects of the theory of pseudo-Riemannian space forms and pseudo-Riemannian geometry. Moreover, we shall describe how to generalise the notion of $r$-harmonicity in this context. 

In Section\link\ref{Sec-results} we shall state our main results and their geometric applications. 

Finally, in Section\link\ref{Sec-proofs}, we shall provide all the proofs. 
\vspace{2mm}

For the sake of completeness, we mention in this introduction that another possible, interesting definition of an $r$-order version of the energy functional, which was proposed by Eells-Sampson and Eells-Lemaire (see \cite{EL83, ES}), is
\begin{equation}\label{r-energy-Eells-Lemaire}
 E_{r}^{ES}(\varphi)= \frac{1}{2} \int_M \, \left\|(d+d^*){^r}\varphi \right \|^2 \,dV \,.
  \end{equation}  
As for a detailed discussion and comparison between definitions \eqref{2s-energia}, \eqref{2s+1-energia} and \eqref{r-energy-Eells-Lemaire}, we refer to \cite{BMOR1}. We believe that most of the techniques used in \cite{BMOR1} could also be applied in the pseudo-Riemannian context, but we shall not pursue this option in this paper.\\
%

{\bf Acknowledgements}. The authors would like to thank Professor Miguel Dom\'{i}nguez V\'{a}zquez for very useful correspondence on Lorentzian isoparametric hypersurfaces.

\section{Pseudo-Riemannian geometry, pseudo-Riemannian space forms and $r$-harmonicity}\label{Sec-preliminaries}
A basic reference for pseudo-Riemannian geometry is the classical book of O'Neill (see \cite{Neill}), but for the specific topics treated in this section we also refer to \cite{Abe, Beem, Dong, Liu, Sasahara, Zhang}.

Let $(M^m_t,g)$ be a pseudo-Riemannian manifold of dimension
$m$ with a \textit{nondegenerate} metric of \textit{index} $t$ ($0 \leq t \leq m$). In order to clarify the notion of index $t$, let us first recall that nondegeneracy means that the only vector $X \in T_pM$ satisfying $g_p(X, Y ) = 0$ for all $Y \in T_pM$ 
is $X = 0$, for any $p \in M $.  A local \textit{orthonormal} frame field
of $(M^m_t,g)$ is a set of local vector fields $\{e_i\}_{i=1}^m$ such that $g(e_i, e_j) = \varepsilon_i \delta_{ij}$, with
$\varepsilon_1=\ldots \varepsilon_t=-1$, $\varepsilon_{t+1}=\ldots \varepsilon_m=1$.

Next, let us fix terminology and notations concerning pseudo-Riemannian space forms.
The $m$-dimensional pseudo-Euclidean space with index $t$ is denoted by $\R^m_t=(\R^m,\langle,\rangle)$, where 
\[
\langle x,y \rangle= -\sum_{i=1}^t x_i y_i +\sum_{i=t+1}^m x_i y_i \,.
\]
The $m$-dimensional pseudo-Riemannian sphere, denoted by $\s^m_t (c)$ is defined as follows:
\begin{equation}\label{eq-def-pseudosfere}
\s^m_t (c)= \left \{x \in \R^{m+1}_t \colon \langle x,x \rangle = \frac{1}{c} \right \} \quad \quad (c>0)\,.
\end{equation}
$\s^m_t (c)$, with the induced metric from $\R^{m+1}_t$, is a complete pseudo-Riemannian manifold with index $t$ and constant positive sectional curvature $c$.

The $m$-dimensional pseudo-Riemannian hyperbolic space, denoted by $\H^m_t (c)$ is defined as follows:
\begin{equation}\label{eq-def-pseudohyperbolic}
\H^m_t (c)= \left \{x \in \R^{m+1}_{t+1} \colon \langle x,x \rangle =\frac{1}{c} \right \} \quad \quad (c<0)\,.
\end{equation}
$\H^m_t (c)$, with the induced metric from $\R^{m+1}_{t+1}$, is a complete pseudo-Riemannian manifold with index $t$ and constant negative sectional curvature $c$.

A pseudo-Riemannian space form refers to one of the three spaces $\R^m_t$, $\s^m_t (c),\H^m_t (c)$. We shall write $\s^m_t $ and $\H^m_t $ for $\s^m_t (1)$ and $\H^m_t (-1)$ respectively. Sometimes, to provide a unified treatment, we also use $N^m_t(c)$ to denote a pseudo-Riemannian space form of sectional curvature $c$.
 
The flat ($c=0$) pseudo-Riemannian space $\R^m_t$ is called \textit{Minkowski space}, while $\s^m_t (c)$ and $\H^m_t (c)$ are known as \textit{de Sitter space} and \textit{anti-de Sitter space} respectively. When the index is $t=1$, these spaces are also referred to as \textit{Lorentz space forms}. We also point out that $\s^m_t (c)$ is diffeomorphic to $\R^t \times \s^{m-t}$, while $\H^m_t (c)$ is diffeomorphic to $\s^{t} \times \R^{m-t}$. In particular, $\s^m_{m-1} (c)$ and $\H^m_1 (c)$ are not simply connected. 

Finally, we point out that in this paper we restrict our study to \textit{connected} manifolds. Therefore, even if in some definitions formally one has two connected components (for instance, according to \eqref{eq-def-pseudosfere}, \eqref{eq-def-pseudohyperbolic}, this happens to $\s^m_m (c)$ and $\H^m_0 (c)$), we shall implicitly assume that we just work with one connected component.

Next, we recall for future use that the sectional curvature tensor field of $N^m_t(c)$ is described by the following simple expression:
\begin{equation}\label{tensor-curvature-N(c)}
R^{N(c)}(X,Y)Z=c\, \big(\langle Y,Z \rangle X-\langle X,Z \rangle Y \big) \quad \quad  \forall \,X,Y,Z \in C(TN(c)) \,.
\end{equation}
Pseudo-Riemannian space forms have important applications in the theory of general relativity and $\R^3_1$, $\s^3_1 (c)$ and $\H^3_1 (c)$ are model spaces for Minkowski, de Sitter and  anti-de Sitter space-time respectively.

\vspace{3mm}

Now, in order to make this paper as self-contained as possible, we follow \cite{Dong} and recall here how the basic operators of Riemannian geometry extend to the pseudo-Riemannian context. In a chart of $(M^m_t,g)$, for a local orthonormal frame field $\{e_i\}_{i=1}^m$ we have the following basic identities, which hold for any $X \in {C(TM)}$, $f \in {C^\infty(M)}$ and bilinear form $b \in C( \odot^2 TM)$:
\begin{eqnarray}\label{key-formulas-pseudo}\nonumber
X &=&\sum_{i=1}^m \varepsilon_i g(X,e_i) e_i \,;\\\nonumber
\grad f &=&\sum_{i=1}^m \varepsilon_i df(e_i) e_i  \,;\\
{\rm div} X &=&\sum_{i=1}^m \varepsilon_i g(\nabla_{e_i}X,e_i)  \,;\\\nonumber
\Delta f &=&-\sum_{i=1}^m \varepsilon_i \left [e_i(e_if)-\big (\nabla_{e_i}e_i\big)f\right ]  \,;\\\nonumber
\trace b& =&\sum_{i=1}^m \varepsilon_i b(e_i,e_i)\,.
\end{eqnarray}
In this paper we shall focus on the study of \textit{pseudo-Riemannian} (or, \textit{nondegenerate}) hypersurfaces of a pseudo-Riemannian space form. We can describe such hypersurfaces by means of a smooth  
$\varphi:M^m_{t'} \to N^{m+1}_t$ and the hypothesis that $M^m_{t'}$ is pseudo-Riemannian simply means that the metric induced by $\varphi$ is nondegenerate. Therefore we can assume that, locally, there always exists an orthonormal frame field $\{e_i \}_{i=1}^m$ on $M^m_{t'}$. Moreover, denoting $\eta$ a unit normal vector field, $\{e_i ,\eta\} $ is a local orthonormal frame field on $N^{m+1}_t$ and we have two possibilities: either $\langle \eta, \eta \rangle=\varepsilon=1$ and $t'=t$, or $\varepsilon=-1$ and $t'=t-1$. 

In particular, when the ambient space is Lorentzian and $\varepsilon=-1$, then the hypersurface is Riemannian and the standard terminology is to say that it is \textit{space-like}. Let us denote $B$ the second fundamental form and $A_\eta=A$ the associated shape operator. The classical Gauss and Weingarten formulas hold as in the Riemannian case, i.e.,
\begin{eqnarray}\label{Gauss-Weingarten-eq}
\nabla^N_X Y&=&\nabla^M_X Y +B(X,Y) {\,;}\\\nonumber
\nabla^N_X \eta&=& - A(X)
\end{eqnarray}
for all tangent vector fields to $M^m_{t'}$. We observe that it is easy to deduce from \eqref{Gauss-Weingarten-eq} that 
\begin{equation}\label{eq-legame-A-B}
B(X,Y)=  \varepsilon \left < A(X),Y\right > \eta\,.
\end{equation}
The mean curvature vector field of $M^m_{t'}$, denoted ${\mathbf H}$, is defined by
\[
{\mathbf H}= \frac{1}{m}\,\trace B = f\,\eta 
\]
where, using \eqref{eq-legame-A-B}, we deduce that the mean curvature function $f$ is given by
\begin{equation}\label{mean-curv-fuction-pseudo}
f=  \frac{1}{m} \varepsilon \sum_{i=1}^m \varepsilon_i \langle A(e_i),e_i \rangle \,.
\end{equation}
In this paper, we shall always say that a hypersurface of a pseudo-Riemannian space is \textit{minimal} if $f$ vanishes identically. However, we mention that, when the ambient space is Lorentzian and the hypersurface is space-like, the term \textit{maximal} instead of minimal is also used in the literature, because in this case small regions are local maximizers of the volume functional.
 
For future use, we recall that
\begin{equation}\label{norma-A}
\trace A^2=\,\sum_{i=1}^m  \varepsilon_i \langle A(A(e_i)),e_i  \rangle=\sum_{i=1}^m  \varepsilon_i \langle A(e_i),A(e_i)\rangle\,.
\end{equation}
We point that in \cite{Dong, Liu} the right hand side of \eqref{norma-A} is denoted by $\|A\|^2$ instead of $\trace A^2$.

%
\vspace{3mm}

In the last equality of \eqref{norma-A} we used the fact that the shape operator $A$ is self-adjoint. We point out that, in general, $A$ is \textit{not} diagonalizable. Indeed, since $\langle \,,\rangle$ is not positive definite, it is well-known that complex eigenvalues may appear or, in some cases, $A$ may not even be diagonalizable over $\C$ (see \cite{Hahn}). We shall study these situations in detail in Section\link\ref{Sec-results}.

Next, we introduce an important family of pseudo-Riemannian hypersurfaces of $N=\s^{m+1}_t$ ($m \geq2$, $1 \leq t \leq m$) and give, in Table~\ref{Hypersurfaces-pseudo-sphere}, the expression for $A$ and $\varepsilon$ (see \cite{Abe}). The definition of the hypersurfaces  in Table~\ref{Hypersurfaces-pseudo-sphere} are given in Table~\ref{Hypersurfaces-pseudo-sphere-notations1} and Table~\ref{Hypersurfaces-pseudo-sphere-notations2}.
%
%
\begin{table}[h!]
\begin{tabular}{ |p{3.5cm}|p{6cm}|p{1.5cm}|  }
 \hline
 \multicolumn{3}{|c|}{Pseudo-Riemannian hypersurfaces of $\s^{m+1}_t=N\subset \R^{m+2}_t $} \\
 \hline
 Hypersurface & Shape operator $A$  & $\varepsilon$\\
 \hline\vspace{.5mm}
 
 $\s^m_t(c) $   &\vspace{.5mm} $\pm \sqrt{c-1}\,I$    &\vspace{.5mm} 1\vspace{.5mm}\\
 $\s^m_{t-1}(c)$ &  $\pm \sqrt {1-c} \,I$  & $-1$   \vspace{1mm}\\
$\R^m_{t-1}$ &$\pm \,I $ & $-1$\vspace{1mm}\\
$\H^m_{t-1}(c)$    &$\pm \sqrt {1-c}\, I$ & $-1$\vspace{1mm}\\
 $\s^{k}_{\ell}(c) \times \s^{m-k}_{t-\ell}\left (\frac{c}{c-1}\right )$ &   $\pm(\sqrt{c-1}\, I_k \oplus - \sqrt {1/(c-1)}\, I_{m-k})$  & $1$\vspace{1mm}\\
 $\s^{k}_{\ell}(c) \times \H^{m-k}_{t-\ell-1}\left (\frac{c}{c-1} \right )$ & $\pm(\sqrt{1-c}\, I_k \oplus  \sqrt {1/(1-c)}\, I_{m-k}) $  & $-1$   \\
 \hline
\end{tabular}
\vspace{1mm}

\caption{}
\label{Hypersurfaces-pseudo-sphere}
\end{table}

\begin{table}[h!]
\begin{tabular}{ |p{1.cm}p{.2cm}p{8.1cm}|p{1.6cm}|  }
 \hline
$\s^m_t(c)$ & =&$\big\{ x=(x_1,\ldots,x_{m+2}) \in N \colon x_{m+2}=\sqrt{1-(1/c)}\big\}$& $1 \leq c$\vspace{1mm}
\\
$\s^m_{t-1}(c)$&=&$\big \{ x\in N \colon x_1=\sqrt{(1/c)-1} \big \}$ & $0<c\leq 1$\vspace{1mm}\\
$\R^m_{t-1}$&=&$\big \{ x\in N \colon x_{1}=x_{m+2}+a\big \}$ & $a>0$\vspace{1mm}\\
$\H^m_{t-1}(c)$&=&$\big\{ x\in N \colon x_{m+2}=\sqrt{1-(1/c)}\big \}$ & $ c<0$\\
 \hline
\end{tabular}
\vspace{1mm}

\caption{}
\label{Hypersurfaces-pseudo-sphere-notations1}
\end{table}

{\small
\begin{table}[h!]
\begin{tabular}{ |p{3.1cm}p{.1cm}p{5.6cm}|p{2.4cm}|  }
 \hline
$\displaystyle{\s^{k}_{\ell}(c) \times \s^{m-k}_{t-\ell} (\frac{c}{c-1} )}$ &=&$\displaystyle{ \Big\{ x\in N \colon -\sum_{i=1}^{\ell}x_i^2+\sum_{i=t +1}^{t+k-\ell+1}x_i^2= \frac{1}{c}},$ &$c>1$ \\
&&$ \displaystyle{-\sum_{i=\ell+1}^{t}x_i^2+\sum_{i=t+k-\ell +2}^{m+2}x_i^2= \frac{c-1}{c} \Big\}}$ &$1\leq k \leq m-1$\\
&&&$0\leq \ell \leq k$\\
&&&$0\leq \ell \leq t$\\
&&&$t-\ell \leq m-k$\\
 \hline
 $\displaystyle{\s^{k}_{\ell}(c) \times \H^{m-k}_{t-\ell-1} (\frac{c}{c-1}  )}$ &=&$\displaystyle{ \Big\{ x\in N \colon -\sum_{i=1}^{\ell}x_i^2+\sum_{i=t +1}^{t+k-\ell+1}x_i^2= \frac{1}{c}},$ &$1>c>0$ \\
&&$ \displaystyle{-\sum_{i=\ell+1}^{t}x_i^2+\sum_{i=t+k-\ell +2}^{m+2}x_i^2= \frac{c-1}{c} \Big\}}$ &$1\leq k \leq m-1$\\
&&&$0\leq \ell \leq k$\\
&&&$0\leq \ell \leq t-1$\\
&&&$t-\ell-1 \leq m-k$\\
 \hline
\end{tabular}
\vspace{1mm}
\caption{}
\label{Hypersurfaces-pseudo-sphere-notations2}
\end{table}
}

The importance of the families of hypersurfaces described in  Table~\ref{Hypersurfaces-pseudo-sphere} lays in the fact {that} essentially any pseudo-Riemannian hypersurface of $\s^{m+1}_t$ with diagonalizable shape operator having at most two distinct constant principal curvatures is locally congruent to one of these (see Theorem 5.1 of \cite{Abe}, and also Theorem\link\ref{Th-nostro-rigidity-2-curv}, for more details).

\begin{remark}\label{Rem-basta-c>0} 
A similar family of hypersurfaces of pseudo-hyperbolic spaces is available, but we omit its description because all the properties of $r$-harmonic submanifolds in $\H^{m+1}_{t^*}$ can be deduced from those of $r$-harmonic submanifolds in $\s^{m+1}_{t}$, with $t=m+1-t^*$. This is a consequence of the fact that, up to a multiplicative constant factor -1 in the metric (that determines a {non-isometric} transformation), irrelevant for the $r$-harmonic equation, we can identify $\s^{m+1}_t$ with $\H^{m+1}_{m+1-t}$. For this reason, without loss of generality, in this paper all the existence and classification results for $r$-harmonic submanifolds that do not depend on a given signature of the metric will be stated only for the cases that the curvature of the ambient space is either $c=1$ or $c=0$.

%
\end{remark}
%
%
%

Next, to prepare the ground for the study of $r$-harmonicity, we consider a general smooth map $\varphi:(M^m_{t'},g) \to (N^n_t,h)$ between two pseudo-Riemannian manifolds. Let us denote $\nabla^{\varphi}$ the induced connection on the pull-back bundle $\varphi ^{-1}TN$. In the pseudo-Riemannian context, the operator corresponding to the classical Riemannian \textit{rough Laplacian} on sections of $\varphi^{-1}  TN$, which will still be denoted $\overline{\Delta}$, becomes
\begin{equation} \label{roughlaplacian-pseudo}
    \overline{\Delta}=d^* d =-\sum_{i=1}^m\varepsilon_i\left(\nabla^{\varphi}_{e_i}
    \nabla^{\varphi}_{e_i}-\nabla^{\varphi}_
    {\nabla^M_{e_i}e_i}\right)\,,
\end{equation}
where again $\{e_i\}_{i=1}^m$ is a local orthonormal frame field tangent to $M^m_{t'}$. 

We are now in the right position to summarise the key points which enable us to describe the generalisation of the notions of harmonicity and, more generally, $r$-harmonicity to the pseudo-Riemannian context.  

A smooth map $\varphi:(M^m_{t'},g) \to (N^n_t,h)$ between two pseudo-Riemannian manifolds is \textit{harmonic} if its tension field vanishes identically, i.e.,
\begin{equation}\label{tension-field-pseudo}
\tau(\varphi)=\trace \nabla d \varphi= \sum_{i=1}^m \varepsilon_i \left [ \nabla^{\varphi}_{e_i}d\varphi(e_i)-d \varphi \left (\nabla^M_{e_i} e_i \right )\right ] =0\,,
\end{equation}
where $\{e_i \}_{i=1}^m$ is a local orthonormal frame field on $M^m_{t'}$ as above. As for papers and examples in this context, we cite \cite{Konderak, Ratto}.

Similarly, taking into account \eqref{roughlaplacian-pseudo}, we can define the $r$-energy for a map between two pseudo-Riemannian manifolds precisely as in \eqref{2s-energia}, \eqref{2s+1-energia}. In the Riemannian case, the explicit expression for the $r$-tension field associated to $E_r(\varphi)$ was obtained by Maeta and Wang (see \cite{Maeta1, Wang}). Essentially, in the pseudo-Riemannian context, the only relevant difference appears when one has to take a trace (for instance, see the last equation in \eqref{key-formulas-pseudo}). More specifically, we have the following expression for the $r$-tension field, where $\overline{\Delta}$ and $\tau$ are given in \eqref{roughlaplacian-pseudo} and \eqref{tension-field-pseudo} respectively:
\begin{eqnarray}\label{2s-tension}
\tau_{2s}(\varphi)&=&\overline{\Delta}^{2s-1}\tau(\varphi)-\varepsilon_i R^N \left(\overline{\Delta}^{2s-2} \tau(\varphi), d \varphi (e_i)\right ) d \varphi (e_i) \nonumber\\ 
&&  -\, \sum_{\ell=1}^{s-1}\, \left \{\varepsilon_i R^N \left( \nabla^\varphi_{e_i}\,\overline{\Delta}^{s+\ell-2} \tau(\varphi), \overline{\Delta}^{s-\ell-1} \tau(\varphi)\right ) d \varphi (e_i)  \right .\\ \nonumber
&& \qquad \qquad  -\, \left . \varepsilon_i R^N \left( \overline{\Delta}^{s+\ell-2} \tau(\varphi),\nabla^\varphi_{e_i}\, \overline{\Delta}^{s-\ell-1} \tau(\varphi)\right ) d \varphi (e_i)  \right \} \,\, ,
\end{eqnarray}
where $\overline{\Delta}^{-1}=0$ and $\{e_i\}_{i=1}^m$ is a local orthonormal frame field tangent to $M^m_{t'}$ (the sum over $i$ is not written but understood). Similarly,
\begin{eqnarray}\label{2s+1-tension}
\tau_{2s+1}(\varphi)&=&\overline{\Delta}^{2s}\tau(\varphi)-\varepsilon_i R^N \left(\overline{\Delta}^{2s-1} \tau(\varphi), d \varphi (e_i)\right ) d \varphi (e_i)\nonumber \\ 
&&  -\, \sum_{\ell=1}^{s-1}\, \left \{\varepsilon_i R^N \left( \nabla^\varphi_{e_i}\,\overline{\Delta}^{s+\ell-1} \tau(\varphi), \overline{\Delta}^{s-\ell-1} \tau(\varphi)\right ) d \varphi (e_i)  \right .\\ \nonumber
&& \qquad \qquad  -\, \left . \varepsilon_i R^N \left( \overline{\Delta}^{s+\ell-1} \tau(\varphi),\nabla^\varphi_{e_i}\, \overline{\Delta}^{s-\ell-1} \tau(\varphi)\right ) d \varphi (e_i)  \right \} \\ \nonumber
&& \,-\,\varepsilon_i R^N \Big( \nabla^\varphi_{e_i}\,\overline{\Delta}^{s-1} \tau(\varphi), \overline{\Delta}^{s-1} \tau(\varphi)\Big ) d \varphi (e_i)\,\,. 
\end{eqnarray}
From the analytic point of view, one of the major differences with respect to the Riemannian case is the fact the PDE's system $\tau_r(\varphi)=0$ is \textit{not elliptic} when $1 \leq t' \leq m-1$.
\section{Statement of the results}\label{Sec-results}
In the Riemannian case, when the ambient space form has nonpositive sectional curvature there are several results which assert that, under suitable conditions, an $r$-harmonic submanifold is minimal (see \cite{Chen}, \cite{Maeta1}, \cite{Maeta4} and \cite{Na-Ura}, for instance). Things drastically change when the ambient is the Euclidean sphere $\s^{m+1}$. Indeed, in this case several examples of proper $r$-harmonic hypersurfaces have been constructed and studied (see \cite{BMOR1, Maeta2, MOR-space forms, Mont-Ratto4} and references therein). 

In the pseudo-Riemannian setting there are in the literature several interesting results, but most of them are limited to the biharmonic case. One of the instances which has attracted more attention is the study of space-like biharmonic hypersurfaces in a Lorentzian space form. In this case, it seems that the sign of the curvature of the ambient produces phenomena which are, in some sense, dual with respect to the Riemannian case. In other words, positive curvature increases the rigidity of space-like biharmonic hypersurfaces. For example, the following interesting result was proved by Ouyang:
\begin{theorem}\label{Th-Ouyang} \cite{Ouyang} Let $M^m$ be a CMC space-like biharmonic hypersurface in either $\R^m_1$ or $\s^{m+1}_1$. Then $M^m$ is minimal.
\end{theorem}
In this order of ideas, we also have:
\begin{theorem}\label{Th-Zhang} \cite{Zhang} Let $M$ be a complete, space-like biharmonic surface in $\R^3_1$ or $\s^{3}_1$. Then $M$ must be totally geodesic, i.e., $\R^2$ or $\s^2$.
\end{theorem}
This trend is confirmed in the $r$-harmonic case ($r \geq 3$), as we shall show in Corollary\link\ref{Cor-no-space-like}.

By contrast, when the index of the hypersurface is positive, some results in the biharmonic case are available but their interpretation is less evident. For the purpose of comparison with the results of this paper, we report here, using our notations, the following interesting result of Liu and Du (see Theorem\link1.2 of \cite{Liu}, where the cases $c\neq 0,1$ are also dealt with explicitly):
\begin{theorem}\label{Th-Liu-Dong} \cite{Liu} Let $M^m_{t'}$ be a pseudo-Riemannian proper biharmonic hypersurface in $N^{m+1}_t(c)$. If $M^m_{t'}$ has diagonalizable shape operator with at most two distinct principal curvatures, then $c\neq 0$. Furthermore, when $c =1$, then $t'=t$ and $M^m_{t'}$ is congruent to either $\s^m_t(2)$ or  $\s^{m_1}_{t_1}(2) \times \s^{m-m_1}_{t-t_1}( 2 )$ with $m_1 \neq m-m_1$. 
%
%
%
\end{theorem}
\begin{remark} When $c=1$, our statement of Theorem\link\ref{Th-Liu-Dong} is equivalent to Theorem\link1.2 of \cite{Liu}, but our formulation makes it easier the comparison with the classical results for biharmonic hypersurfaces in the Riemannian case. 
\end{remark}
The results of \cite{Liu} in the biharmonic case were recently refined in \cite{Dong}. Finally, for the sake of completeness, we also have to cite the interesting paper \cite{Sasahara} by Sasahara, where the author classifies proper biharmonic curves and surfaces in de Sitter 3-space and anti-de Sitter 3-space. In Theorem\link\ref{Th-Sasah-r>2} below, we shall extend this classification for surfaces in $N^3_1(c)$ to the case $r\geq3$.

Our investigation of this type of problems in the pseudo-Riemannian context starts with the following general result, which is a pseudo-Riemannian version of Theorem\link\ref{Th-existence-hypersurfaces-c>0-e-c<0}. To the purpose of a quick comparison between Theorems\link\ref{Th-existence-hypersurfaces-c>0-e-c<0} and \ref{Th-existence-hypersurfaces-c>0-e-c<0-pseudo}, we point out that, if the induced metric on the hypersurface is Riemannian, then $\trace A^2=||A||^2$. 
\begin{theorem}\label{Th-existence-hypersurfaces-c>0-e-c<0-pseudo} Assume that $m \geq2$ and $1 \leq t \leq m$. Let $M^m_{t'}$ be a non-minimal CMC pseudo-Riemannian hypersurface in a pseudo-Riemannian space form $N^{m+1}_t(c)$ and assume that $\trace A^2$ is constant. Then
$M^m_{t'}$ is proper biharmonic  if and only if 
\begin{equation}\label{2-harmonicity-condition-in-general-pseudo}
\varepsilon \trace A^2-m\,c=0 \,.
\end{equation}
If $r \geq 3$, then
$M^m_{t'}$ is proper $r$-harmonic if and only if either
\begin{equation}\label{eq-Tr-A^2=0}
\trace A^2=0
\end{equation}
or
\begin{equation}\label{r-harmonicity-condition-in-general-pseudo}
\varepsilon \left (\trace A^2\right )^2-m\,c\,\left (\trace A^2\right )-(r-2)m^2 \,c\, \alpha^2=0 \,,
\end{equation}
where $\alpha$ denotes the mean curvature of $M^m_{t'}$ and $\varepsilon=\langle \eta,\eta\rangle$.
\end{theorem} 
\begin{remark} The special case $r=2$ in Theorem\link\ref{Th-existence-hypersurfaces-c>0-e-c<0-pseudo} is part of the result proved in \cite{Liu}, where the condition of biharmonicity for a general hypersurface $M^m_{t'} \hookrightarrow N^{m+1}_t(c)$ was computed. In the case of surfaces \eqref{2-harmonicity-condition-in-general-pseudo} was first obtained in \cite{Sasahara}.
\end{remark}
\begin{remark} It is important to point out that in Theorem\link\ref{Th-existence-hypersurfaces-c>0-e-c<0-pseudo} we do not require that the shape operator $A$ be diagonalizable. By way of example, in Theorem\link\ref{Th-Xiao-c=-1} we shall exhibit a new family of $r$-harmonic surfaces in $\H^3_1$ whose shape operator is \textit{not} diagonalizable.
\end{remark}
A first, immediate consequence of Theorem\link\ref{Th-existence-hypersurfaces-c>0-e-c<0-pseudo} is the following

\begin{corollary}\label{Cor-no-space-like} Assume $r \geq 3$ and $m\geq2$. Let $M^m$ be a space-like, $r$-harmonic CMC hypersurface in a Lorentzian space form $N^{m+1}_1(c)$. If $\trace A^2$ is constant and $c \geq 0$, then $M^m$ is minimal.
\end{corollary}
In the same spirit:

\begin{corollary}\label{Cor-no-space-like-bis} Assume that $r \geq 3$, $m \geq2$ and $1 \leq t \leq m$. Let $M^m_{t'}$ be a pseudo-Riemannian $r$-harmonic CMC hypersurface in $N^{m+1}_t(c)$. If $\trace A^2$ is a positive constant and $\varepsilon c < 0$, then $M^m_{t'}$ is minimal.
\end{corollary}

Next, we shall use Theorem\link\ref{Th-existence-hypersurfaces-c>0-e-c<0-pseudo} to construct new examples of $r$-harmonic hypersurfaces. Indeed,
\begin{corollary}\label{Cor-r-harmonic-pseudo-hyperspheres} Assume that $r \geq 3$, $m \geq2$ and $1 \leq t \leq m$. Let $\s^m_{t}(c)$ be a small pseudo-hypersphere in $\s^{m+1}_t$. Then $\s^m_{t}(c)$ is proper $r$-harmonic if and only if $c=r$.
\end{corollary}
With the notation of Tables\link\eqref{Hypersurfaces-pseudo-sphere}-\eqref{Hypersurfaces-pseudo-sphere-notations1}-\eqref{Hypersurfaces-pseudo-sphere-notations2}:
\begin{theorem}\label{Th-r-harmonic-pseudo-Clifford} Assume that $r \geq 3$, $m \geq2$ and $1 \leq t \leq m$. Let $M^m_t=\s^{k}_{\ell}(c) \times \s^{m-k}_{t-\ell}\left (\frac{c}{c-1}\right )$ $(c>1)$ be a generalised pseudo-Clifford torus in $\s^{m+1}_t$. Then $M^m_t$ is proper $r$-harmonic if and only if $c \neq \frac{m}{k}$ and
\begin{equation}\label{condizione-pseudo-tori-Clifford}
P_3(c)=k c^3  - k (r+2) c^2  + [m (r-1) + k (r+2)] c  - m r=0\,.
\end{equation}
\end{theorem}
\begin{remark} The third order polynomial $P_3(c)$ in \eqref{condizione-pseudo-tori-Clifford} is equivalent to the one which appears in the Riemannian case. To see this, it is enough to set
\[
k=p \,\,; \quad m-k=q \,\,; \quad c=\frac{1}{x} \,.
\]
Then, up to a constant factor $-1/x^3$, $P_3(c)$ becomes the third order polynomial obtained in Theorem\link1.2 of \cite{Mont-Ratto4}. We point out that, in the special case that $k=m-k$, we have
\[
P_3(c)=(-2 + c) k (c^2 + r - c r) \,.
\]
Therefore, when $k=m-k$, the pseudo-Clifford torus $\s^{k}_{\ell}(c) \times \s^{m-k}_{t-\ell}\left (\frac{c}{c-1}\right )$ is proper $r$-harmonic in $\s^{m+1}_t$ if and only if
\begin{equation}\label{eq-k=m-k}
 c=\frac{r \pm \sqrt{r^2-4r}}{2} \qquad \,\, (r \geq 5) \,.
 \end{equation} 
For a more detailed discussion on the existence and qualitative behaviour of admissible roots of $P_3(c)$ in the general case we refer to \cite{Mont-Ratto4}.
%

\end{remark}
Next, we show that the examples given in Corollary\link\ref{Cor-r-harmonic-pseudo-hyperspheres} and Theorem\link\ref{Th-r-harmonic-pseudo-Clifford} are the only possible ones within a certain class of hypersurfaces. More precisely, we prove the following result:
\begin{theorem}\label{Th-nostro-rigidity-2-curv} Assume that $r \geq 3$, $m \geq2$ and $1 \leq t \leq m$. Let $M^{m}_{t'}$ be a pseudo-Riemannian hypersurface with diagonalizable shape operator in a pseudo-Riemannian space form $N^{m+1}_t (c)$, $c=0$ or $c=1$. Assume that  there exist  at most two distinct principal curvatures and that they are constant on $M^{m}_{t'}$. If $M^{m}_{t'}$ is proper $r$-harmonic, then $M^{m}_{t'}$ is one of the examples given in Corollary\link\ref{Cor-r-harmonic-pseudo-hyperspheres} or Theorem\link\ref{Th-r-harmonic-pseudo-Clifford}.
\end{theorem} 
\begin{remark} Of course, it is straightforward to state explicitly the version of Corollary\link\ref{Cor-r-harmonic-pseudo-hyperspheres}, Theorem\link\ref{Th-r-harmonic-pseudo-Clifford} and Theorem\link\ref{Th-nostro-rigidity-2-curv} in the case that $c=-1$. However, for the reasons explained in Remark\link\ref{Rem-basta-c>0}, we omit the details.
\end{remark}
Next, we obtain some geometric results for triharmonic surfaces which provide a version in the pseudo-Riemannian case of some facts which we proved in \cite{MOR-space forms} in the Riemannian case.
\begin{theorem}\label{Th-non-existence-surfaces-c<=0}
Assume $r \geq 3$. Let $M^2_{t'}$ be a pseudo-Riemannian triharmonic CMC surface in $N^{3}_t(c)$ and assume that its shape operator $A$ is diagonalizable. If $\varepsilon c \leq 0$, then $M^2_{t'}$ is minimal.
\end{theorem}  
Next, we focus on the case that the ambient space is $\s^{3}_t$. Our result is:
\begin{theorem}\label{Th-structure-surfaces-c>0}
Let $M^2_{t'}$ be a CMC proper triharmonic surface in $\s^3_t$ and assume that its shape operator $A$ is diagonalizable. Then $M^2_{t'}$ is an open part of the small pseudo-hypersphere $\s^2_t(3)$.
\end{theorem} 
%

A widely studied family of hypersurfaces in the pseudo-Riemannian setting is that of isoparametric Lorentzian hypersurfaces of a Lorentzian space-form. For this specific topic, we refer to \cite{Magid, Xiao} and, for further background, to \cite{Hahn, Nomizu-pseudo}.
We recall that a Lorentzian hypersurface $M^m_1$ in a Lorentzian space form $N^{m+1}_1(c)$ is said to be \textit{isoparametric} if the minimal polynomial of the shape operator $A$ is constant on $M^m_1$. 
We know by \cite[Proposition~2.1]{Hahn}  that  $M^m_1$  has constant principal curvatures with constant algebraic multiplicities. 

Moreover, according to \cite[Chapter~9]{Neill} there exist bases where the shape operator $A$ assumes one of the following Jordan canonical forms:
\\

$
{\rm I}\quad \begin{bmatrix}
a_1&\cdots&0\\
\vdots& \ddots&\vdots\\
0&\cdots& a_m
\end{bmatrix}
\quad\hfill
{\rm II}\quad \begin{bmatrix}
a_0&0&&&\\
1&a_0&&&&\\
&&a_1&&\\
&&&\ddots&\\
&&&&a_{m-2}
\end{bmatrix}
$

$
{\rm III}\quad \begin{bmatrix}
a_0&0&0&&&\\
0&a_0&1&&&\\
-1&0&a_0&&&\\
&&&a_1&&\\
&&&&\ddots&\\
&&&&&a_{m-3}
\end{bmatrix}
\quad\hfill 
{\rm IV}\quad\begin{bmatrix}
a_0&b_0&&&\\
-b_0&a_0&&&&\\
&&a_1&&\\
&&&\ddots&\\
&&&&a_{m-2}
\end{bmatrix}
$
\\

Here $b_0$ is assumed to be non-zero.  In cases I, II and III the eigenvalues are
real, while $a_0 \pm i b_0$ are complex eigenvalues in case IV. A 
Lorentzian isoparametric hypersurface in $N^{m+1}_1(c)$ is called of type I, II, III or IV  according to the form of its shape operator $A$.

We observe that a Lorentzian isoparametric hypersurface in $N^{m+1}_1(c)$ is a CMC hypersurface with $\trace A^2$ constant. Therefore, Theorem\link\ref{Th-existence-hypersurfaces-c>0-e-c<0-pseudo} applies to this type of hypersurfaces.  

Now, a direct computation shows that for a Lorentzian isoparametric hypersurface of type
I, II and III, with $\trace A\neq 0$, we have  $\trace A^2>0$. Differently, for hypersurfaces of type IV it is possible to have $\trace A\neq 0$ and $\trace A^2\leq0$. If the curvature of the ambient space is $c=0$, we easily deduce from Theorem~\ref{Th-existence-hypersurfaces-c>0-e-c<0-pseudo} that a non-minimal Lorentzian isoparametric hypersurface is $r$-harmonic ($r\geq 2$) if and only if $\trace A^2= 0$.
These observations prove our first result in this context, that is
\begin{proposition}\label{Th-magid} Assume $r \geq 2$ and $m\geq2$. Let $M^m_1$ be an isoparametric Lorentzian $r$-harmonic hypersurface in the flat Lorentzian space form $\R^{m+1}_1$. Then either $M^m_1$ is minimal or its shape operator is of type IV.
\end{proposition}

In \cite[Theorem~4.10]{Magid} Magid gave a proof that there exist no isoparametric Lorentzian hypersurfaces of type IV in the flat Lorentzian space form $\R^{m+1}_1$. However, it was pointed out in \cite{RamosVasquezLopez} that there could be some gaps in the arguments of \cite{Magid}.

By contrast, in the case that the curvature of the ambient space is $c=-1$, we shall construct some new examples of proper $r$-harmonic Lorentzian isoparametric surfaces in $\H^{3}_1$ with non-diagonalizable shape operator of type IV.
More precisely, first, as above, we observe that Theorem\link\ref{Th-existence-hypersurfaces-c>0-e-c<0-pseudo} implies that a Lorentzian isoparametric hypersurface in $\H^{m+1}_1$ can be $r$-harmonic only if its shape operator is non-diagonalizable and of type IV. 

The following result proves that {\em complex circles} (see \cite{Magid2}) provide a geometrically significant example of such hypersurfaces: 

\begin{theorem}\label{Th-Xiao-c=-1}   Let $M^2_1$ be a complex circle in $\H^{3}_1\subset \R^4_2$ parametrized by 
\[
\begin{split}
x(s,t)= &\big\{b \cos (s) \cosh (t)-a \sin (s)
   \sinh (t),\\
   &a \cos (s) \sinh (t)+b
   \sin (s) \cosh (t),\\
   &a \cos (s)
   \cosh (t)+b \sin (s) \sinh (t),\\
   &b
   \cos (s) \sinh (t)-a \sin (s)
   \cosh (t)\big\}\,,
   \end{split}
\]
where $a$ and $b$ are real numbers such that $b^2-a^2=1$ and $ab\neq 0$.
%
Then $M^2_1$ is proper $r$-harmonic provided that either
\begin{enumerate}
\item  $r>2$ and
\[
{a^2=\frac{\sqrt{2}}{2}-\frac{1}{2}\,,\quad 
b^2=\frac{\sqrt{2}}{2}+\frac{1}{2}}
\]
 or
\item $r=3$ and
\[
 a^2=\frac{\sqrt{3}}{3}-\frac{1}{2}\,, \quad  b^2=\frac{\sqrt{3}}{3}+\frac{1}{2}\,.
\]
\end{enumerate}
\end{theorem}
\begin{remark} The shape operator $A$ of the $r$-harmonic complex circles obtained in Theorem\link\ref{Th-Xiao-c=-1} is non-diagonalizable and of type IV. As we shall see in the proof of Theorem\link\ref{Th-Xiao-c=-1}, the instances of Case $(1)$ have $\trace A^2=0$, while those of Case $(2)$ have $\trace A^2<0$. In accordance with the result of Sasahara (see \cite[Theorem~5.4]{Sasahara}), the family of isoparametric surfaces studied in Theorem\link\ref{Th-Xiao-c=-1} does not contain any proper biharmonic immersion.
\end{remark}

Next, we recall that a \textit{null curve} $\gamma(s)$ in $N^3_1(c) \subset \R^4_t$ is a smooth curve such that $\langle \gamma'(s), \gamma'(s)\rangle \equiv 0$. A {\em $B$-scroll} over a null curve $\gamma(s)$ is a surface of index $1$ in $N^3_1(c)$ parametrized by
\begin{equation}\label{eq-B-Scroll}
x(s,u)=\gamma(s)+u\,B(s) \,,
\end{equation}
where $\{A,B,C \}$ is a pseudo-orthonormal frame field, or a \textit{Cartan frame field}, along $\gamma(s)$, i.e.,
\begin{equation}\label{eq-Cartan-frame}
\begin{split}
\langle A,A\rangle &=\langle B,B\rangle =0\,, \qquad \langle A,B\rangle=-1 \,, \\ 
\langle A,C\rangle &=\langle B,C\rangle =0\,, \qquad \langle C,C\rangle=1\\
&{\rm and} \\
\gamma'(s)&=A(s) \,,\\
C'(s)&=- \lambda \, A(s)- k(s)\,B(s) \,,
\end{split}
\end{equation}
where $\lambda$ is a real constant and $k(s) \neq 0$ (see \cite{Alias, Hahn, Sasahara} for details).

We obtain a version of \cite[Theorem~5.4]{Sasahara} in the case that $r \geq3$. More precisely, we shall prove:
\begin{theorem}\label{Th-Sasah-r>2} Assume that $r \geq3$. Let $M^2_{t'}$ be an isoparametric pseudo-Riemannian surface in a $3$-dimensional Lorentz space form $N^3_1(c)$, where $c \in \{-1,1\}$. 
Then $M^2_{t'}$ is proper $r$-harmonic if and only if it is congruent to an open subset of one of the following:
\begin{itemize}
\item[(1)] $\s^2_1(r) \subset \s^3_1$;
\item[(2)] $\H^2(-r) \subset \H^3_1$;
\item[(3)] $\s^{1}_{\ell}(c) \times \s^{1}_{1-\ell}\left (\frac{c}{c-1}\right ) \subset \s^3_1$, where $0 \leq \ell \leq 1$ and 
\begin{equation*}\label{eq-condiz-c}
c=\frac{r \pm \sqrt{r^2-4r}}{2} \qquad \,\, {(r \geq 5)} \,;
\end{equation*}
\item[(4)] $\H^{1}_{1-\ell}(-c) \times \H^{1}_{\ell}\left (-\,\frac{c}{c-1}\right ) \subset \H^3_1$, where $0 \leq \ell \leq 1$ and $c$ is as in Case~{\rm (3)};
\item[(5)] a $B$-scroll over a null curve in $\s^3_1$ whose Gauss curvature $K$ is constant and equal to $r$;
\item[(6)] an $r$-harmonic complex circle in $\H^3_1$ of the type described in Theorem\link\ref{Th-Xiao-c=-1}.
\end{itemize}
\end{theorem}
\begin{remark} Comparison of \cite[Theorem~5.4]{Sasahara} with our Theorem\link\ref{Th-Sasah-r>2} shows that the relevant differences between the biharmonic case and the case $r \geq 3$ is the appearance in Theorem\link\ref{Th-Sasah-r>2} of the family of solutions of type $(6)$ for all $r \geq3$, and of type $(3)$ and $(4)$ for $r \geq 5$.  These facts, together with Theorem\link\ref{Th-Xiao-c=-1}, support the following general idea: we can find geometrically interesting situations where there is no biharmonic instance, but there exist examples of $r$-harmonic immersions, $r \geq 3$. Moreover, comparing the result of the present paper with those of \cite{MOR-space forms}, we see that the notion of $r$-harmonicity is more flexible in the pseudo-Riemannian setup compared to the Riemannian case as it allows for a larger class of solutions. In this order of ideas, we also cite the phenomenon illustrated in Remark\link\ref{Remark-appendix} below.
\end{remark}

\begin{example}\label{Example-B-scroll} First, we point out that the family of $r$-harmonic surfaces in $\s^3_1$ obtained in Case (5) of Theorem\link\ref{Th-Sasah-r>2} is very ample. This is a consequence of the fact that, given any real constant $\lambda$ and smooth function $k(s)$, it is always possible to determine (at least locally) a null curve $\gamma$ such that its associated $B$-scroll \eqref{eq-B-Scroll} verifies \eqref{eq-Cartan-frame}. As proved in \cite{Kim}, the existence of such a null curve $\gamma$ can be deduced by solving a suitable Cauchy problem for a first order linear system of ordinary differential equations. In the proof of Theorem\link\ref{Th-Sasah-r>2} we shall show that $K=\lambda^2+1$ and any such surface is proper $r$-harmonic provided that $\lambda^2=r-1\, (r \geq 2)$. 

As a special case, following the procedure described in the Appendix of \cite{Kim}, here we give the explicit expression of the null curve $\gamma(s)$ and of its associated vector field $B(s)$ assuming that $k(s) \equiv 1$. Note that when, as in our example, the function $k(s)$ is bounded, $\gamma(s)$ is defined on the whole $\R$.  To this end, let 
{\footnotesize 
\[
\begin{split}
\gamma(s)=&\Big\{
\frac{\sqrt{c} (c+d-2) \sin(\sqrt{d} s)+2 c \cos
   (\sqrt{d} s)+\sqrt{d} (c+d+2) \sinh (\sqrt{c}
   s)+2 d \cosh(\sqrt{c} s)}{2 (c+d)},\\
   &
   \frac{
   \sqrt{c} (c+d) \sin(\sqrt{d} s)+2 c \cos(\sqrt{d}
   s)+ \sqrt{d} (c+d) \sinh(\sqrt{c} s)+2d
   \cosh(\sqrt{c} s)}{2(c+d)},\\
   &
   \frac{-\sqrt{c} \sin(\sqrt{d} s)+c \cos(\sqrt{d} s)+\sqrt{d} \sinh(\sqrt{c} s)+d \cosh(\sqrt{c}
   s)}{c+d},\\
   &
   \frac{\cos(\sqrt{d} s)-\cosh(\sqrt{c}
   s)}{c+d}\Big\}
   \end{split}
\]
}
where
\[
c=\sqrt{1+\lambda^2}+\lambda\,,\quad d=\sqrt{1+\lambda^2}-\lambda\,.
\]
Then $\gamma(s)$ is a null curve in $\s^3_1$ with associated vector field 
{\footnotesize 
\[
\begin{split}
B(s)=&\frac{1}{4}\Big\{2 \sqrt{c} \sin(\sqrt{d} s)+(2-c-d)
   \cos(\sqrt{d} s)+2 \sqrt{d} \sinh(\sqrt{c}
   s)+(2+c+d) \cosh(\sqrt{c} s) ,\\
   &
   2 \sqrt{c} \sin(\sqrt{d} s)+2 \sqrt{d} \sinh(\sqrt{c} s)+(c+d) \left(\cosh(\sqrt{c} s)-\cos(\sqrt{d} s)\right),\\
   &
    2\sqrt{c} \sin(\sqrt{d} s)+2\sqrt{d} \sinh(\sqrt{c} s)+2\cosh(\sqrt{c} s)+2\cos(\sqrt{d}
   s),\\
   &
   2\sqrt{d} \sin(\sqrt{d}
   s)-2\sqrt{c} \sinh(\sqrt{c} s)\Big\}
\end{split}
\]
}
and the parametrized surface 
\[
x(s,u)=\gamma(s)+u\,B(s)
\]
defines a $B$-scroll with Gauss curvature $K=1+\lambda^2$ and $k(s)\equiv1$. Choosing $\lambda=\pm\sqrt{r-1}$ we obtain the desired $r$-harmonic surface in $\s^3_1$.

For the sake of completeness, the method of computation used to determine the curve $\gamma(s)$ and the vector field $B(s)$ will be given at the end of this paper, in Appendix\link\ref{Appendix}.
\end{example}

\section{Proofs}\label{Sec-proofs}
As a preliminary step, in the following lemma we state without proof some standard facts which we shall use in this section.
\begin{lemma}\label{Lemma-tecnico-1} Let $\varphi:M^m_{t'} \to  N^{m+1}_t(c)$ be a pseudo-Riemannian hypersurface. Let $A$ denote the shape operator and $f=(1/m)\varepsilon  \trace A$ the mean curvature function. Then
\begin{itemize}
\item[{\rm (a)}] $(\nabla A) (\cdot,\cdot)$ is symmetric;
\item[{\rm (b)}] $\langle (\nabla A) (\cdot,\cdot), \cdot \rangle$ is totally symmetric;
\item[{\rm (c)}] $\trace (\nabla A) (\cdot,\cdot)= m \,\varepsilon\,\grad f$. 
\end{itemize}
\end{lemma}
Next, we perform our first computation:
\begin{lemma}\label{Lemma-Delta-H} Let $\varphi:M^m_{t'} \to  N^{m+1}_t(c)$ be a pseudo-Riemannian hypersurface and denote by $\eta$ the unit normal vector field. Then
\begin{equation}\label{Delta-H-formula}
 \overline{\Delta} {\mathbf H}= (\Delta f + f \varepsilon \trace A^2) \eta + 2 A(\grad f ) + m f \varepsilon \grad f \,.
\end{equation}
\end{lemma}
\begin{proof} We work with a geodesic frame field $\left \{ X_i \right \}_{i=1}^m$ around an arbitrarily fixed point $p \in M^m_{t'}$. Also, we simplify the notation writing $\nabla$ for $\nabla^{M}$. Since ${\mathbf H}=f \eta$, around $p$ we have:
\[
\nabla_{X_i}^{\varphi} {\mathbf H}=\nabla_{X_i}^{\perp} {\mathbf H}-A_{{\mathbf H}}(X_i)=\left ( X_i f \right )\eta -f A  \left (X_i \right )\,.
\]
Then at $p$ we have:
\begin{eqnarray*}
\nabla_{X_i}^{\varphi}\nabla_{X_i}^{\varphi} {\mathbf H}&=&\left(X_i X_i f \right )\eta- \left( X_i f \right )A  \left (X_i \right )\\
&&- \left( X_i f \right )A  \left (X_i \right )-f \big( \nabla_{X_i} A  \left (X_i \right )+B\left( X_i, A  \left (X_i \right )\right )\big )\nonumber\\
&=&\left(X_i X_i f \right )\eta- 2\left( X_i f \right )A  \left (X_i \right )-f (\nabla A)(X_i,X_i)\\
&&-f\varepsilon \langle A(X_i),  A(X_i)\rangle  \eta \,,
\label{eq:nabla2H}
\end{eqnarray*}
where, for the second equality, we also used \eqref{eq-legame-A-B}.
Now we take the sum over $i$ as in \eqref{roughlaplacian-pseudo} and, using \eqref{key-formulas-pseudo} and Lemma~\ref{Lemma-tecnico-1}, we obtain \eqref{Delta-H-formula} (note that the sign convention for $\Delta$ and $\overline{\Delta}$ is given in \eqref{key-formulas-pseudo}, \eqref{roughlaplacian-pseudo}).
\end{proof}
Next, we assume that the mean curvature function $f$ is constant and we obtain
\begin{lemma}\label{Lemma-Delta2-H} Let $\varphi:M^m_{t'} \to  N^{m+1}_t(c)$ be a pseudo-Riemannian hypersurface and assume that its mean curvature function $f$ is equal to a constant $\alpha$. Then
\begin{equation}\label{Delta-2-H-formula}
 \overline{\Delta}^2 {\mathbf H}= \alpha \left (\varepsilon\Delta \trace A^2+ (\trace A^2)^2 \right ) \eta + 2 \alpha \varepsilon A \big(\grad \trace A^2 \big )  \,.
\end{equation}
\end{lemma}
\begin{proof} Since $f$ is constant, according to Lemma\link\ref{Lemma-Delta-H} we have $\overline{\Delta}{\mathbf H}=\alpha \varepsilon \trace A^2 \eta$ and so
\[
\overline{\Delta}^2{\mathbf H}=\alpha \varepsilon\overline{\Delta}\big ( \trace A^2 \eta \big)\,.
\]
Now, around $p$:
\[
\nabla_{X_i}^{\varphi} \big ( \trace A^2 \eta\big )=\big ( X_i \trace A^2 \big )\eta -\trace A^2  A\left (X_i \right )\,.
\]
At $p$:
\begin{eqnarray*}
\nabla_{X_i}^{\varphi}\nabla_{X_i}^{\varphi} \big ( \trace A^2 \eta\big )&=&\left(X_i X_i \trace A^2 \right )\eta- \left( X_i \trace A^2 \right )A  \left (X_i \right )\\
&&- \left( X_i \trace A^2 \right )A  \left (X_i \right )-\trace A^2 \big( \nabla_{X_i} A  \left (X_i \right )+B\left( X_i, A  \left (X_i \right )\right )\big )\,.\\
\end{eqnarray*}
Next, computing as in Lemma\link\ref{Lemma-Delta-H}, we find
\[
\overline{\Delta}\big ( \trace A^2 \eta \big)=\left (\Delta \trace A^2 \right ) \eta+ 2  A \big(\grad \trace A^2\big )  +m \trace A^2 \varepsilon\grad f + \varepsilon (\trace A^2)^2 \eta
\]
and, since $f$ is constant, the proof ends immediately.
\end{proof}
We are now in the right position to prove our first theorem.

\begin{proof}[Proof of Theorem\link\ref{Th-existence-hypersurfaces-c>0-e-c<0-pseudo}] As $\trace A^2$ is constant, it follows from Lemma\link\ref{Lemma-Delta2-H} that
\begin{equation}\label{Delta-H-A2-constant}
\overline{\Delta}^2 {\mathbf H}= \alpha (\trace A^2)^2  \eta\,.
\end{equation}
Now, from the computations above, we observe that
\begin{equation}\label{Delta-eta}
\overline{\Delta}\eta=\varepsilon \trace A^2 \eta \,.
\end{equation}
Next, putting together \eqref{Delta-H-formula}, \eqref{Delta-H-A2-constant} and \eqref{Delta-eta}, we easily deduce that
\begin{equation}\label{Delta-H-potenza-p}
\overline{\Delta}^p {\mathbf H}= \alpha \,\varepsilon ^p \left(\trace A^2\right)^{p}  \eta \quad \quad \forall p \in \n^* \,.
\end{equation}
Now we are in a good position to perform the explicit calculation of the $r$-tension field $\tau_{r}(\varphi)$ described in \eqref{2s-tension}, \eqref{2s+1-tension}. We begin with $\tau_{2s}(\varphi)$, $s \geq2$. Using \eqref{Delta-H-potenza-p}, \eqref{tensor-curvature-N(c)} and computing we obtain (as in \eqref{2s-tension}, we omit to write the sum over $i$):

\begin{eqnarray*}
\frac{1}{m} \tau_{2s}(\varphi)&=&\alpha \varepsilon\left(\trace A^2\right)^{2s-1} \eta-c  \,\Big \{ \varepsilon_i\langle d\varphi(X_i),d\varphi(X_i)   \rangle \alpha \left(\trace A^2\right)^{2s-2} \eta \\
&&  - \varepsilon_i \langle d\varphi(X_i), \alpha \left(\trace A^2\right)^{2s-2} \eta   \rangle  d\varphi(X_i) \Big \}\\
&& -c\,  m\varepsilon \sum_{\ell=1}^{s-1}  \Big \{ \varepsilon_i\langle d\varphi(X_i), \alpha \left(\trace A^2\right)^{s-\ell-1} \eta   \rangle \big (  -\alpha \left(\trace A^2\right)^{s+\ell-2} A(X_i) \big )  \\
&& - \varepsilon_i\langle d\varphi(X_i), -\alpha \left(\trace A^2\right)^{s+\ell-2} A(X_i)   \rangle \alpha \left(\trace A^2\right)^{s-\ell-1} \eta \Big \}\\
&&+c\,  m \varepsilon\sum_{\ell=1}^{s-1}  \Big \{\varepsilon_i \langle d\varphi(X_i),- \alpha \left(\trace A^2\right)^{s-\ell-1} A(X_i)   \rangle  \alpha \left(\trace A^2\right)^{s+\ell-2} \eta \\
&& - \varepsilon_i\langle d\varphi(X_i), \alpha \left(\trace A^2\right)^{s+\ell-2} \eta   \rangle  \big (  -\alpha \left(\trace A^2\right)^{s-\ell-1} A(X_i) \big )   \Big \}\\
&=&\alpha\varepsilon \left(\trace A^2\right)^{2s-1} \eta-c \, m \alpha \left(\trace A^2\right)^{2s-2} \eta \\
&& -c\, \varepsilon m \Big \{ \sum_{\ell=1}^{s-1} \big [\varepsilon m \alpha^3 \left(\trace A^2\right)^{2s-3} \eta \big ] + \sum_{\ell=1}^{s-1} \big [ \varepsilon m \alpha^3 \left(\trace A^2\right)^{2s-3} \eta \big ] \Big \}\\
&=&\alpha \left(\trace A^2\right)^{2s-3} \Big \{\varepsilon \left(\trace A^2\right)^2-m \,c\, \trace A^2- (2s-2)m^2\,c\,\alpha^2 \Big \}\eta\,.
\end{eqnarray*}

This completes the proof in the case $r=2s$. The case $r=2s+1$ is similar and so we omit the details.
\end{proof}
\begin{proof}[Proof of Corollary\link\ref{Cor-no-space-like}] $\,$

Since a space-like hypersurface is Riemannian, $\trace A^2=||A||^2$. Then, either the hypersurface is totally geodesic, or it follows from \eqref{r-harmonicity-condition-in-general-pseudo} with $\varepsilon=-1$ that $M^m$ cannot be proper $r$-harmonic, thus the only possibility is that  $M^m$ is minimal, that is $\alpha=0$.
\end{proof}
\begin{proof}[Proof of Corollary\link\ref{Cor-no-space-like-bis}] Since, by assumption, $\trace A^2$ is a positive constant, condition \eqref{eq-Tr-A^2=0} does not hold. Similarly, since $\varepsilon$ and $c$ have opposite sign, equation \eqref{r-harmonicity-condition-in-general-pseudo} cannot be verified. Therefore, the only possibility is that the hypersurface is minimal.
\end{proof}
\begin{proof}[Proof of Corollary\link\ref{Cor-r-harmonic-pseudo-hyperspheres}] We use \eqref{mean-curv-fuction-pseudo}, \eqref{norma-A} and Table~\ref{Hypersurfaces-pseudo-sphere} to compute $\trace A^2=m \,(c-1)$ and $\alpha=\sqrt{c-1}$. Then the thesis follows by direct substitution in \eqref{r-harmonicity-condition-in-general-pseudo}, where the curvature of the ambient space is $1$ and also $\varepsilon=1$.
\end{proof}
\begin{proof}[Proof of Theorem\link\ref{Th-r-harmonic-pseudo-Clifford}] The hypotheses of Theorem\link\ref{Th-existence-hypersurfaces-c>0-e-c<0-pseudo} are verified and so the condition for $r$-harmonicity is equation \eqref{r-harmonicity-condition-in-general-pseudo} where, as in the proof of Corollary\link\ref{Cor-r-harmonic-pseudo-hyperspheres}, the curvature of the ambient space is $1$ and $\varepsilon=1$. Then we compute using the explicit expression for the shape operator given in Table~\ref{Hypersurfaces-pseudo-sphere}:
\begin{eqnarray*}
\trace A^2&=& k(c-1)+\frac{(m-k)}{(c-1)}\,; \\
\alpha^2&=&\frac{1}{m^2}\left (k \sqrt{c-1}-\frac{(m-k)}{ \sqrt{c-1}} \right )^2\\
&=&\frac{(c k - m)^2}{m^2 (c-1)} \,.
\end{eqnarray*}
Next, after direct substitution and simplification, we find that equation \eqref{r-harmonicity-condition-in-general-pseudo} is equivalent to
\[
(ck-m) P_3(c)=0 \,,
\]
where $P_3(c)$ is the third order polynomial defined in \eqref{condizione-pseudo-tori-Clifford}. Now, since $ ck-m=0$ corresponds to $\alpha=0$, the conclusion of the theorem follows immediately.
\end{proof}
\begin{proof}[Proof of Theorem\link\ref{Th-nostro-rigidity-2-curv}] When the curvature of the ambient is $c=1$, according to Theorem 5.1 of \cite{Abe} $M^m_{t'}$ is one of the hypersurfaces listed in Table~\ref{Hypersurfaces-pseudo-sphere}. Then a case by case direct inspection of \eqref{r-harmonicity-condition-in-general-pseudo}, using again \eqref{mean-curv-fuction-pseudo}, \eqref{norma-A} and Table\link\ref{Hypersurfaces-pseudo-sphere} to compute $\trace A^2$ and $\alpha$, shows that the only $r$-harmonic hypersurfaces in this family are those given in Corollary\link\ref{Cor-r-harmonic-pseudo-hyperspheres} and Theorem\link\ref{Th-r-harmonic-pseudo-Clifford}. In the case that $c=0$, Theorem 5.1 of \cite{Abe} says that $M^m_{t'}$ is one of the hypersurfaces listed in (R-1)--(R-6), p. 131 of \cite{Abe}. Then, again, the thesis follows easily by direct inspection.
\end{proof}
\begin{proof}[Proof of Theorem\link\ref{Th-non-existence-surfaces-c<=0}] The $3$-tension field is described by \eqref{2s+1-tension} with $s=1$. In the first part of the proof, for future reference, we do not make any assumption on the dimension $m$ and the curvature $c$. We observe that $\tau(\varphi)=m{\mathbf H}$ and use Lemma\link\ref{Lemma-Delta-H} with $f$ constant and \eqref{tensor-curvature-N(c)}. We have: 
\begin{eqnarray}\label{Curv-tensor-expr-1}\nonumber
\sum_{i=1}^m \varepsilon_i R^{N(c)} \left( \overline{\Delta}\tau(\varphi),d \varphi(X_i)\right ) d \varphi (X_i)&=&c\, m \sum_{i=1}^m \varepsilon_i \Big \{ \langle  d \varphi (X_i), d \varphi (X_i) \rangle  \overline{\Delta}{\mathbf H}\\\nonumber
&&\qquad \quad - \langle  d \varphi (X_i), \overline{\Delta}{\mathbf H}  \rangle d \varphi (X_i) \Big \}\\
&=&c\, m \big \{ m \alpha \varepsilon\,\trace A^2\, \eta - 0 \big \}\\\nonumber
&=&c\, m^2  \alpha \varepsilon\,\trace A^2 \,\eta \,.
\end{eqnarray}
Similarly, we compute 
\begin{equation}\label{Curv-tensor-expr-2}
\sum_{i=1}^m \varepsilon_i R^{N(c)} \left( \nabla^{\varphi}_{X_i} \tau(\varphi),\tau(\varphi) \right ) d \varphi (X_i)=c\, \varepsilon m^3   \alpha^3  \eta \,.
\end{equation}
Using \eqref{Delta-2-H-formula}, \eqref{Curv-tensor-expr-1} and \eqref{Curv-tensor-expr-2} into \eqref{2s+1-tension} and Lemma\link\ref{Lemma-Delta2-H} we obtain the explicit expression of the $3$-tension field:
\begin{eqnarray*}\label{tri-tension-field-explicit}
\tau_3(\varphi)&=&m \alpha \big[\varepsilon\Delta \trace A^2 + \left (\trace A^2\right)^2 -m\, c \,\varepsilon\,\trace A^2 -m^2\, c\, \varepsilon \, \alpha^2 \big ] \eta \\\nonumber
&&+ 2m\alpha \varepsilon\, A \left ( \grad \trace A^2 \right ) 
\end{eqnarray*}
Therefore, we conclude that $M^m_{t'}$ is a triharmonic hypersurface in $N^{m+1}_t(c)$ if and only if either it is minimal or
\begin{equation}\label{tri-harmonicity-system-explicit}
\left \{
\begin{array}{ll}
{\rm (i)}\quad &\Delta \trace A^2 +  \varepsilon\, \left ( \trace A^2\right)^2 -m\, c \,\trace A^2 -m^2\, c\, \alpha^2=0 \\
\\
{\rm (ii)}\quad & A \left ( \grad \trace A^2\right ) =0 \,.
\end{array}
\right .
\end{equation}
From now on, we assume that $M^m_t$ is not minimal and we use the hypothesis $\varepsilon c \leq 0$. First, we analyse the case $c=0$. If $\trace A^2$ is constant, then it follows immediately from \eqref{tri-harmonicity-system-explicit} that $ \left ( \trace A^2\right)^2=0$ and so, since $A$ is diagonalizable, the hypersurface is totally geodesic, a contradiction. If $\trace A^2$ is not a constant, then there exists an open set $U$ of the surface such that $\grad \trace A^2\neq 0$ on $U$.  We deduce from \eqref{tri-harmonicity-system-explicit}(ii) that $0$ is an eigenvalue of $A$ on $U$. Now we use the assumption that $m=2$. Since $M^2_{t'}$ is CMC, it is easy to conclude that necessarily $\trace A^2$ is again a constant on $U$, and this is a contradiction. Next, if $c\neq0$, by the assumption $\varepsilon c \leq 0$ we deduce that $c$ and $\varepsilon$ have opposite sign. Then, using this observation in \eqref{tri-harmonicity-system-explicit} (i), the proof follows essentially the same argument as in the case $c=0$.
\end{proof}
\begin{proof}[Proof of Theorem\link\ref{Th-structure-surfaces-c>0}] First, from an argument similar to the proof of Theorem\link\ref{Th-non-existence-surfaces-c<=0}, we deduce that $\trace A^2$ must be a constant on $M^2_{t'}$ and also the two principal curvatures are constant. 

Then, according to Theorem\link\ref{Th-nostro-rigidity-2-curv}, we deduce that the only possibility is that $M^2_{t'}$ is an open part of a small pseudo-hypersphere $\s^2_t(3)$ because, according to Theorem\link\ref{Th-r-harmonic-pseudo-Clifford}, in these dimensions there exists no generalised pseudo-Clifford torus which is proper triharmonic. Indeed, using $k=1, r=3$ and $m=2$, \eqref{condizione-pseudo-tori-Clifford} becomes
\[
P_3(c)=c^3-5 c^2+9 c-6\,.
\]
Now, the only real root of $P_3(c)$ is $c=2$ and it corresponds to the minimal pseudo-Clifford torus.
\end{proof}
\begin{proof}[Proof of Theorem\link\ref{Th-Xiao-c=-1}] 

We have to find under what conditions on $a,b$ the complex circle $x(s,t)$ is $r$-harmonic. By using standard techniques of the theory of surfaces in $\H^{3}_1\subset \R^4_2$, we can compute the shape operator of the complex circle $x(s,t)$ and obtain
\[
A=\frac{1}{a^2+b^2}\begin{bmatrix}
 2 ab & 1 \\
 &\\
 -1& 2 ab  \\
\end{bmatrix}\,.
\]
%
%

The condition $ab\neq 0$ ensures that $\trace A\neq 0$, that is $M^2_1$ is not minimal. 

The case $r=2$. From Theorem~\ref{Th-existence-hypersurfaces-c>0-e-c<0-pseudo} we deduce that $M^2_1$ is biharmonic if and only if $\trace A^2+2=0$. Now, a direct computation gives
\[
\trace A^2+2= \frac{16 b^2 (b^2-1)}{(a^2+b^2)^2}\,,
\]
from which, since $b^2-1=a^2$, we deduce that $M^2_1$ cannot be proper biharmonic. 

The case $r>2$. In this case, using Theorem~\ref{Th-existence-hypersurfaces-c>0-e-c<0-pseudo}, we have two possibilities, that is:
(i) either
\[
0=\trace A^2=2 \,\frac{4a^2b^2-1}{(a^2+b^2)^2}\,;
\]
(ii) or $4a^2b^2-1\neq 0$ and, according to \eqref{r-harmonicity-condition-in-general-pseudo},
\[
\begin{split}
0&=\left(\trace A^2\right )^2+2\,\left (\trace A^2\right )+(r-2) (\trace A)^2\\
&=\frac{16 b^2 \left(b^2-1\right)
   \left(\left(a^2+b^2\right)^2
   r-4\right)}{\left(a^2+b^2\right)^4
   }\,.
\end{split}
\]

Case (i), together with the hypothesis $b^2-a^2=1$, give exactly  point (1) of the theorem.  As for Case (ii), we have a proper solution if and only if 
\[
1<a^2+b^2=\frac{2}{\sqrt{r}}
\]
from which we deduce that the only possible value is $r=3$. Taking into account that, by assumption, $b^2-a^2=1$, it is easy to conclude Case (2) of the theorem.
\end{proof}
\begin{proof}[Proof of Theorem\link\ref{Th-Sasah-r>2}] Case (i). Here we assume that the shape operator $A$ is diagonalizable. Since the hypersurface has dimension $m=2$ and constant principal curvatures, we can apply Theorem\link\ref{Th-nostro-rigidity-2-curv}. Moreover, taking into account Remark\link\ref{Rem-basta-c>0}, the proper $r$-harmonic surfaces in $\H^3_1$ can be deduced from those of $\s^3_2$. Putting these facts together and using \eqref{eq-k=m-k} we obtain Cases {$(1)-(4)$} in the statement of the theorem.

Case (ii). Now, we assume that the shape operator $A$ is non-diagonalizable. In this case, $t'=1$ and so $\varepsilon =1$. If $A$ has a double, real eigenvalue $\lambda$, then the argument of \cite[Proposition\link4.1]{Alias} enables us to conclude that $M^2_1$ is a $B$-scroll in $N^3_1(c)$ over a null curve $\gamma$, and its shape operator $A$, with respect to the coordinate frame field $\left \{ \partial \slash \partial s,\,\partial \slash \partial u \right \}$ associated to \eqref{eq-B-Scroll}, is given by
\[
A= \left [
\begin{array}{cc}
\lambda &0 \\
k(s)&\lambda
\end{array}
\right ]
\]
with $k(s) \neq 0$. Therefore, $\trace A^2= 2 \lambda^2$ and $\alpha= \lambda$.
Next, we assume that the $B$-scroll is not minimal, i.e., $\lambda \neq 0$, and we apply \eqref{r-harmonicity-condition-in-general-pseudo}: if  $c=-1,\,m=2, \,\varepsilon=1$ then there are no solutions; if $c=1,\,m=2, \,\varepsilon=1$, then it is easy to deduce that the $B$-scroll is proper $r$-harmonic if and only if
\[
\lambda^2 =r-1 \,.
\]
Finally, observing that the Gauss curvature $K$ of the $B$-scroll is given by $K=\det (A)+c$, with $c=1$, we conclude Case (5). 

It remains to analyse the case that the shape operator $A$ has two complex eigenvalues. This will lead us to Case (6). Indeed, from \cite[p.~453]{Alias}, we conclude that $M^2_1$ is a flat Lorentzian surface in $\H^3_1$ with parallel second fundamental form in $\R^4_2$. Then, according to \cite{Magid2}, $M^2_1$ is locally congruent to a complex circle and therefore the conclusion follows from Theorem~\ref{Th-Xiao-c=-1}.
\end{proof}

\appendix
\section{} \label{Appendix}
In this appendix we describe how to construct the $B$-scrolls in $\s^3_1 \subset \R^4_1$ described in Example\link\ref{Example-B-scroll}. We follow exactly the method given in the Appendix of \cite{Kim}. 

First, we introduce the following matrices:
\begin{equation}\label{Eq-matrici-E-T-appendix}
 E=\left [
 \begin{array}{rccc}
 -1 &0&0&0\\
 0&1&0&0\\
  0&0&1&0\\
   0&0&0&1
  \end{array}
 \right ]\,, \,\,\quad 
 T=\left [
 \begin{array}{rrcc}
 0&-1&0&0\\
 -1&0&0&0\\
  0&0&1&0\\
   0&0&0&1
  \end{array}
 \right ]\,. 
 \end{equation} 
 Next, we introduce a $4 \times 4$ matrix $X(s)$:
\begin{equation*}\label{eq-X-appendix}
X(s)= [A(s),B(s),C(s),\gamma(s)]\,,
\end{equation*} 
where the first $3$ columns will be the Cartan frame field $\{A,B,C \}$ along $\gamma$. Let
\begin{equation*}\label{eq-M-appendix}
 M(s)=\left [
 \begin{array}{rrrc}
 0&0&-\lambda&1\\
 0&0&-k(s)&0\\
  -k(s)&-\lambda&0&0\\
   0&1&0&0
  \end{array}
 \right ] 
\end{equation*}
and $X(s)$ be the unique solution of the Cauchy problem associated to
\begin{equation}\label{eq-diff-eq-X-appendix}
X'(s)=X(s)\,M(s) 
\end{equation}
with initial condition $X(0)=[A(0),B(0),C(0),\gamma(0)]$ satisfying
\begin{equation}\label{eq-cond-iniz}
X(0){}^t\,E \,X(0) =T \,,
\end{equation}
where the matrices $E,T$ were defined in \eqref{Eq-matrici-E-T-appendix}. As shown in \cite{Kim}, we conclude that \eqref{eq-Cartan-frame} holds and $\{A(s),B(s),C(s) \}$ is a Cartan frame field along the null curve $\gamma(s)$ in $ \s^3_1  \subset \R^4_1$. Finally, we set
\[
X(0)=
 \left [
 \begin{array}{cccc}
 1&1&0&1\\
 1&0&0&1\\
  0&1&0&1\\
   0&0&1&0
  \end{array}
 \right ]\,.
\]
It is easy to check that \eqref{eq-cond-iniz} holds. Then, letting $k(s)\equiv 1$ and solving \eqref{eq-diff-eq-X-appendix} with the aid of the software Mathematica${}^{\footnotesize \textregistered} $, we find the explicit $r$-harmonic $B$-scrolls in $\s^3_1$ described at the end of Example\link\ref{Example-B-scroll}. 
\begin{remark}\label{Remark-appendix}Finally, we point out that the method of this Appendix does not require $k(s) \neq 0$. If we assume that the smooth function $k(s)$ vanishes at an isolated point, say $k(s_0)=0$, we have a very interesting phenomenon. The associated $B$-scroll is still a proper $r$-harmonic surface in $\s^3_1$ provided that $\lambda^2=r-1$. Indeed, on the whole surface we have $\trace A= 2 \lambda$ and $\trace A^2= 2 \lambda^2$ and so equation \eqref{r-harmonicity-condition-in-general-pseudo} applies and $r$-harmonicity follows exactly as in the proof of Theorem\link\ref{Th-Sasah-r>2}. But, in this case, the minimal polynomial of $A$ is $(x-\lambda)$ at all points $x(s_0,u)$, while it is $(x-\lambda)^2$ elsewhere. Therefore, these proper $r$-harmonic surfaces in $\s^3_1$ \textit{are not isoparametric}. By contrast, in the Riemannian case, all the known proper $r$-harmonic CMC hypersurfaces with $\trace A^2$ constant are isoparametric (see \cite{MOR-space forms}). 
\end{remark}


\begin{thebibliography}{99}
\bibitem{Abe} N.~Abe, N.~Koike, S.~Yamaguchi. Congruence theorems for proper semi-Riemannian hypersurfaces in a real space form. {\em Yokohama Math. J.} 35 (1987), 123--136.

\bibitem{Alias} L.J.~Alias, A.~Ferrandez, P.~Lucas. $2$-type surfaces in $\s^3_1$ and $\H^3_1$. {\em Tokyo J. Math.} 17 (1994), 447--454.


\bibitem{Beem} J.K.~Beem, P.E.~Ehrlich, K.L.~Easley. Global Lorentzian geometry. Second edition. Monographs and Textbooks in Pure and Applied Mathematics, Marcel Dekker, Inc., New York 202 (1996).

%
%
%
%
\bibitem{Volker} V.~Branding. {The stress-energy tensor for polyharmonic maps}. {\em Nonlinear Anal.} 190 (2020), 111616, 17 pp. 

\bibitem{Volker2}
V.~Branding.
{A structure theorem for polyharmonic maps between Riemannian manifolds}.
{\em J. Differential Equations}, 273 (2021), 14--39.

\bibitem{BMOR1} V.~Branding, S.~Montaldo, C.~Oniciuc, A.~Ratto. {Higher order energy functionals}. {\em Adv. Math.} 370 (2020), 107236, 60 pp.
%
%
\bibitem{CMO02} R.~Caddeo, S.~Montaldo, C.~Oniciuc. Biharmonic submanifolds of $\s^3$. {\em Internat. J. Math.} 12 (2001), 867--876.

\bibitem{CMO03} R.~Caddeo, S.~Montaldo, C.~Oniciuc. Biharmonic submanifolds in spheres. {\em Israel J. Math.} 20 (2002), 1--15.




\bibitem{Chen} B.-Y.~Chen. {Total mean curvature and submanifolds of finite type}. Second edition. Series in Pure Mathematics, 27. World Scientific Publishing Co. Pte. Ltd., Hackensack, NJ, (2015).

%
%

\bibitem{RamosVasquezLopez}
J.C.~D\'iaz-Ramos, M.~Dom\'inguez-V\'azquez, V.~SanMart\'in-L\'opez. Isoparametric hypersurfaces in complex hyperbolic spaces. {\em Adv. Math.}
314 (2017), 756--805.

\bibitem{Dong} Y.~Dong, Y-L.~Ou. Biharmonic submanifolds of pseudo-Riemannian manifolds. {\em J. Geom. Phys.} 112 (2017), 25--262.

\bibitem{EL83} J.~Eells, L.~Lemaire. {Selected topics in harmonic maps.} CBMS Regional Conference Series in Mathematics, 50. American Mathematical Society, Providence, RI, 1983.

\bibitem{EL1} J.~Eells, L.~Lemaire. {Another report on harmonic maps}.
 {\it Bull. London Math. Soc.} 20 (1988), 385--524.

\bibitem{ES} J.~Eells, J.H.~Sampson. {Variational theory in fibre bundles}. {\em Proc. U.S.-Japan Seminar in Differential Geometry}, Kyoto (1965), 22--33.


\bibitem{Hahn} J.~Hahn. Isoparametric Hypersurfaces
in the Pseudo-Riemannian Space Forms. {\em Math. Z.} 187 (1984), 195--208.


\bibitem{Jiang} G.Y.~Jiang. {2-harmonic maps and their first and second variational formulas}.
{\it Chinese Ann. Math. Ser. A }7 (1986), 389--402. Translated from the Chinese by Hajime Urakawa. Note Mat. 28 (2009), 209--232.

\bibitem{Kim}D-S.~Kim, Y-H.~Kim. B-Scrolls with non-diagonalizable shape operators. {\em Rocky Mountain J. Math.} 33 (2003), 175--190.

\bibitem{Konderak} J.J.~Konderak. Construction of harmonic maps between pseudo-Riemannian spheres and hyperbolic spaces. {\em Proc. Amer. Math. Soc.} 109 (1990), 469--476.

\bibitem{Liu} J.~Liu, L.~Du. Classification of proper biharmonic hypersurfaces in pseudo-Riemannian space forms. {\em Differential Geom. Appl.} 41 (2015), 110--122.

\bibitem{Maeta1} S.~Maeta. {k-harmonic maps into a Riemannian manifold with constant sectional curvature}.
{\it Proc. Amer. Math. Soc.} 140 (2012), 1835--1847.

\bibitem{Maeta3} S.~Maeta. {The second variational formula of the $k$-energy and $k$-harmonic curves}. {\it Osaka J. Math.} 49 (2012), 1035--1063.

\bibitem{Maeta2} S.~Maeta. {Construction of triharmonic maps}. {\it Houston J. Math.} 41 (2015), 433--444.

\bibitem{Maeta4} S.~Maeta, N.~Nakauchi, H.~Urakawa. {Triharmonic isometric immersions into a manifold of non-positively constant curvature}. {\it Monatsh. Math.} 177 (2015), 551--567.

\bibitem{Magid2} M.A.~Magid. Isometric immersions of Lorentz space with parallel second fundamental forms. {\em Tsukuba J. Math.} 8 (1984), 31--54.

\bibitem{Magid} M.A.~Magid. Lorentzian isoparametric hypersurfaces. {\em Pac. J. Math.} 118 (1985), 165--197.

\bibitem{MOR-space forms} S.~Montaldo, C.~Oniciuc, A.~Ratto. Polyharmonic hypersurfaces into space forms. arXiv:1912.10790 ({\em Israel J. Math.}, in press).

\bibitem{Mont-Ratto4} S.~Montaldo, A.~Ratto. New examples of $r$-harmonic immersions into the sphere. {\em J. Math. Anal. Appl.} 458 (2018), 849--859.

%
%

\bibitem{Na-Ura} N.~Nakauchi, H.~Urakawa. Polyharmonic maps into the Euclidean space. {\em Note Mat.} 38 (2018), 89--100. 


\bibitem{Nomizu-pseudo} K.~Nomizu. On isoparametric hypersurfaces in the Lorentzian space forms. {\em Japan J. Math.} 7 (1981), 217--226.

\bibitem{Neill} B.~O'Neill. Semi-Riemannian geometry with applications to relativity. {\em Academic Press} Series Pure and Appl. Math. 103 (1983).

\bibitem{SMCO} C.~Oniciuc. {Biharmonic submanifolds in space forms}. Habilitation thesis. DOI:10.13140/2.1.4980.5605, (2012).

\bibitem{Ou}
{ Y.~Ou, B-Y.~Chen}. { Biharmonic submanifolds and biharmonic maps in {R}iemannian
  geometry}. World Scientific (2020).

\bibitem{Ouyang} C.~Ouyang. Biharmonic space-like submanifolds in pseudo-Riemannian space form. {\em Chinese Ann. Math.} Ser. A, 21 (2000), 649--654.

%
%
%
\bibitem{Ratto} A.~Ratto. Equivariant harmonic maps between manifolds with metrics of $(p,q)$-signature. {\em Ann. Inst. H. Poincar\'e} 6 (1989), 503--524.

\bibitem{Sasahara} T.~Sasahara. Biharmonic submanifolds in nonflat Lorentz 3-space forms. {\em Bull. Aust. Math. Soc.} 85 (2012), 422--432.


\bibitem{Xiao} L.~Xiao. Lorentzian isoparametric hypersurfaces in $\H^{n+1}_1$. {\em Pac. J. Math.} 189 (1999), 377--397.

\bibitem{Wang} S.B.~Wang. {The first variation formula for k-harmonic mappings}.
{\it Journal of Nanchang University} 13, N.1 (1989).

\bibitem{Wang2} S.B.~Wang. Some results on stability of 3-harmonic mappings. {\it Chinese Ann. Math.} Ser. A 12 (1991), 459--467.

\bibitem{Zhang}  W.~Zhang. Biharmonic space-like hypersurfaces in pseudo-Riemannian space. https://arxiv.org/abs/0808.1346.

\end{thebibliography}
\end{document}